\documentclass[12pt]{amsart}
\usepackage[nobysame,abbrev,alphabetic]{amsrefs}
\usepackage{amssymb}
\usepackage[all]{xy}

\DeclareMathOperator{\Hom}{Hom}
\DeclareMathOperator{\id}{id}

\DeclareMathOperator{\Infil}{Infil}

\DeclareMathOperator{\invlim}{\varprojlim}

\DeclareMathOperator{\Lyn}{Lyn}

\DeclareMathOperator{\Sh}{Sh}
\DeclareMathOperator{\Shuffles}{Shuffles}

\DeclareMathOperator{\trg}{trg}

\DeclareFontFamily{U}{wncy}{}
\DeclareFontShape{U}{wncy}{m}{n}{<->wncyr10}{}
\DeclareSymbolFont{mcy}{U}{wncy}{m}{n}
\DeclareMathSymbol{\Sha}{\mathord}{mcy}{"58}
\DeclareMathSymbol{\sha}{\mathord}{mcy}{"78}

\begin{document}

\newtheorem{thm}{Theorem}[section]
\newtheorem{cor}[thm]{Corollary}
\newtheorem{lem}[thm]{Lemma}
\newtheorem{prop}[thm]{Proposition}
\newtheorem{defin}[thm]{Definition}
\newtheorem{exam}[thm]{Example}
\newtheorem{examples}[thm]{Examples}
\newtheorem{rem}[thm]{Remark}
\newtheorem{case}{\sl Case}
\newtheorem{claim}{Claim}
\newtheorem{prt}{Part}
\newtheorem*{mainthm}{Main Theorem}
\newtheorem*{thmA}{Theorem A}
\newtheorem*{thmB}{Theorem B}
\newtheorem*{thmC}{Theorem C}
\newtheorem*{thmD}{Theorem D}
\newtheorem{question}[thm]{Question}
\newtheorem*{notation}{Notation}
\swapnumbers
\newtheorem{rems}[thm]{Remarks}
\newtheorem*{acknowledgment}{Acknowledgment}

\newtheorem{questions}[thm]{Questions}
\numberwithin{equation}{section}

\newcommand{\Bock}{\mathrm{Bock}}
\newcommand{\dec}{\mathrm{dec}}
\newcommand{\diam}{\mathrm{diam}}
\newcommand{\dirlim}{\varinjlim}
\newcommand{\discup}{\ \ensuremath{\mathaccent\cdot\cup}}
\newcommand{\divis}{\mathrm{div}}
\newcommand{\gr}{\mathrm{gr}}
\newcommand{\nek}{,\ldots,}
\newcommand{\ind}{\hbox{ind}}
\newcommand{\inv}{^{-1}}
\newcommand{\isom}{\cong}
\newcommand{\Massey}{\mathrm{Massey}}
\newcommand{\ndiv}{\hbox{$\,\not|\,$}}
\newcommand{\nil}{\mathrm{nil}}
\newcommand{\pr}{\mathrm{pr}}
\newcommand{\sep}{\mathrm{sep}}
\newcommand{\tagg}{^{''}}
\newcommand{\tensor}{\otimes}
\newcommand{\alp}{\alpha}
\newcommand{\gam}{\gamma}
\newcommand{\Gam}{\Gamma}
\newcommand{\del}{\delta}
\newcommand{\Del}{\Delta}
\newcommand{\eps}{\epsilon}
\newcommand{\lam}{\lambda}
\newcommand{\Lam}{\Lambda}
\newcommand{\sig}{\sigma}
\newcommand{\Sig}{\Sigma}
\newcommand{\bfA}{\mathbf{A}}
\newcommand{\bfB}{\mathbf{B}}
\newcommand{\bfC}{\mathbf{C}}
\newcommand{\bfF}{\mathbf{F}}
\newcommand{\bfP}{\mathbf{P}}
\newcommand{\bfQ}{\mathbf{Q}}
\newcommand{\bfR}{\mathbf{R}}
\newcommand{\bfS}{\mathbf{S}}
\newcommand{\bfT}{\mathbf{T}}
\newcommand{\bfZ}{\mathbf{Z}}
\newcommand{\dbA}{\mathbb{A}}
\newcommand{\dbC}{\mathbb{C}}
\newcommand{\dbF}{\mathbb{F}}
\newcommand{\dbN}{\mathbb{N}}
\newcommand{\dbQ}{\mathbb{Q}}
\newcommand{\dbR}{\mathbb{R}}
\newcommand{\dbU}{\mathbb{U}}
\newcommand{\dbZ}{\mathbb{Z}}
\newcommand{\grf}{\mathfrak{f}}
\newcommand{\gra}{\mathfrak{a}}
\newcommand{\grA}{\mathfrak{A}}
\newcommand{\grB}{\mathfrak{B}}
\newcommand{\grd}{\mathfrak{d}}
\newcommand{\grh}{\mathfrak{h}}
\newcommand{\grI}{\mathfrak{I}}
\newcommand{\grL}{\mathfrak{L}}
\newcommand{\grm}{\mathfrak{m}}
\newcommand{\grp}{\mathfrak{p}}
\newcommand{\grq}{\mathfrak{q}}
\newcommand{\grR}{\mathfrak{R}}
\newcommand{\calA}{\mathcal{A}}
\newcommand{\calB}{\mathcal{B}}
\newcommand{\calC}{\mathcal{C}}
\newcommand{\calE}{\mathcal{E}}
\newcommand{\calG}{\mathcal{G}}
\newcommand{\calH}{\mathcal{H}}
\newcommand{\calK}{\mathcal{K}}
\newcommand{\calL}{\mathcal{L}}
\newcommand{\calM}{\mathcal{M}}
\newcommand{\calW}{\mathcal{W}}
\newcommand{\calV}{\mathcal{V}}

\makeatletter
\renewcommand{\BibLabel}{%
    \Hy@raisedlink{\hyper@anchorstart{cite.\CurrentBib}\hyper@anchorend}%
    [\thebib]%
}
\makeatother

\title[Cohomology and Lyndon words]{The Cohomology of canonical quotients of free groups and Lyndon words}

\author{ Ido Efrat}
\address{Department of Mathematics\\
Ben-Gurion University of the Negev\\
P.O.\ Box 653, Be'er-Sheva 84105\\
Israel} \email{efrat@math.bgu.ac.il}

\thanks{This work was supported by the Israel Science Foundation (grant No.\ 152/13).}

\keywords{Profinite cohomology, lower $p$-central filtration, Lyndon words, Shuffle relations,  Massey products}

\subjclass[2010]{Primary 12G05, Secondary 20J06,  68R15}

\maketitle

\begin{abstract}
For a prime number $p$ and a free profinite group $S$, let $S^{(n,p)}$ be the $n$th term of its lower $p$-central filtration, and $S^{[n,p]}$ the corresponding quotient.
Using tools from the combinatorics of words, we construct a canonical basis of the cohomology group $H^2(S^{[n,p]},\dbZ/p)$, which we call the Lyndon basis, and use it to obtain structural results on this group.
We show a duality between the Lyndon basis and canonical generators of $S^{(n,p)}/S^{(n+1,p)}$.
We prove that the cohomology group satisfies  shuffle relations, which for small values of $n$ fully describe it.
\end{abstract}

\section{Introduction}
Let $p$ be a fixed prime number.
For a profinite group $G$ one defines the {\bf  lower $p$-central filtration} $G^{(n,p)}$, $n=1,2\nek$ inductively by
\[
G^{(1,p)}=G, \quad G^{(n+1,p)}=(G^{(n,p)})^p[G,G^{(n,p)}].
\]
Thus $G^{(n+1,p)}$ is the closed subgroup of $G$ generated by the powers $h^p$ and commutators $[g,h]=g\inv h\inv gh$, where $g\in G$ and $h\in G^{(n,p)}$.
We also set $G^{[n,p]}=G/G^{(n,p)}$.

\smallskip

Now let $S$ be a free profinite group on the basis $X$, and let $n\geq2$.
Then $S^{[n,p]}$ is a free object in the category of pro-$p$ groups $G$ with $G^{(n,p)}$ trivial.
As with any pro-$p$ group,  the cohomology groups $H^l(S^{[n,p]})=H^l(S^{[n,p]},\dbZ/p)$, $l=1,2$,  capture the main information on generators and relations, respectively, in a minimal presentation of $S^{[n,p]}$.
The group $H^1(S^{[n,p]})$ is just the dual $(S^{[2,p]})^\vee\isom\bigoplus_{x\in X}\dbZ/p$, and it remains to understand $H^2(S^{[n,p]})$.

\smallskip

When $n=2$ the quotient $S^{[2,p]}$ is an elementary abelian $p$-group, and the structure of $H^2(S^{[2,p]})$ is well-known.
Namely, for $p>2$ one has an isomorphism
\[
H^1(S^{[2,p]})\oplus{\textstyle\bigwedge^2} H^1(S^{[2,p]})\xrightarrow{\sim}H^2(S^{[2,p]}),
\]
which is the Bockstein map on the first component, and the cup product on the second component.
Furthermore, taking a basis $\chi_x$, $x\in X$, of $H^1(S^{[2,p]})$ dual to $X$, there is a fundamental duality between $p$th powers and commutators in the presentation of $S$ and Bockstein elements and cup products, respectively, of the $\chi_x$ (see \cite{NeukirchSchmidtWingberg}*{Ch.\ III, \S9} for details).
These facts have numerous applications in Galois theory, ranging from class field theory (\cite{Koch02}, \cite{NeukirchSchmidtWingberg}), the works by Serre and Labute on the pro-$p$ Galois theory of $p$-adic fields (\cite{SerreDemuskin}, \cite{Labute67}), the structure theory of general absolute Galois groups (\cite{MinacSpira96}, \cite{EfratMinac11}), the birational anabelian phenomena (\cite{Bogomolov91}, \cite{Topaz15}), Galois groups with restricted ramification (\cite{Vogel05}, \cite{Schmidt10}), and mild groups (\cite{Labute06}, \cite{Forre11}, \cite{LabuteMinac11}),   to mention only a few of the pioneering works in these areas.

\smallskip

In this paper we generalize the above results from the case $n=2$ to arbitrary $n\geq2$.
Namely, we give a complete description of $H^2(S^{[n,p]})$ in terms of a canonical linear basis of this cohomology group.
This basis is constructed using tools from the {\sl combinatorics of words} -- in particular, the  {\bf Lyndon words} in the alphabet $X$, i.e., words which are lexicographically smaller than all their proper suffixes (for a fixed total order on $X$).
We call it the {\bf Lyndon basis}, and use it to prove several structural results on  $H^2(S^{[n,p]})$, and in particular to compute its size (see below).

\smallskip

The Lyndon basis constructed here can be most naturally described in terms of central extensions, as follows:
For $1\leq s\leq n$ let $\dbU$ denote the group of all unipotent upper-triangular $(s+1)\times(s+1)$-matrices over the ring $\dbZ/p^{n-s+1}$.
There is a central extension
\[
0\to\dbZ/p\to\dbU\to \dbU^{[n,p]}\to1
\]
(Proposition \ref{properties of dbU}).
It corresponds to a cohomology element $\gam_{n,s}\in H^2(\dbU^{[n,p]})$.
For a profinite group $G$ and a continuous homomorphism $\rho\colon G\to\dbU$ we write $\bar\rho\colon G^{[n,p]}\to\dbU^{[n,p]}$ for the induced homomorphism, and $\bar\rho^*\gam_{n,s}$ for the pullback to $H^2(G^{[n,p]})$.
Now for any word $w=(x_1\cdots x_s)$ in the alphabet $X$ we define a homomorphism
$\rho^w\colon S\to \dbU$ by setting the entry $(\rho^w(\sig))_{ij}$ to be the coefficient of the subword $(x_i\cdots x_{j-1})$ in the power series $\Lam(\sig)$, where $\Lam\colon S\to(\dbZ/p^{n-s+1})\langle\langle X\rangle\rangle^\times$ is the {\bf Magnus homomorphism}, defined on the generators $x\in X$ by $\Lam(x)=1+x$ (see \S3 and \S6 for more details).
The Lyndon basis is now given by:

\begin{mainthm}
The pullbacks $\alp_{w,n}=(\bar\rho^w)^*\gam_{n,s}$, where $w$ ranges over all Lyndon words of length $1\leq s\leq n$ in the alphabet $X$, form a linear basis of $H^2(S^{[n,p]})$ over $\dbZ/p$.
\end{mainthm}

We further show a duality between the Lyndon basis and certain canonical elements $\sig_w\in S^{(n,p)}$, with $w$ a Lyndon word of length $\leq n$, generalizing the above mentioned duality in the case $n=2$ (see Corollary \ref{upper-triangular}).

\smallskip

The cohomology elements $\bar\rho^*\gam_{n,s}$ include the Bockstein elements (for $s=1$), the cup products (for $n=s=2$), and more generally, the elements of $n$-fold Massey products in $H^2(G^{[n,p]})$ (for $n=s\geq2$); see Examples \ref{examples of  alp}.
The full spectrum $\bar\rho^*\gam_{n,s}$, $1\leq s\leq n$, appears to give new significant ``external" objects in profinite cohomology, which to our knowledge have not been investigated so far in general.

\smallskip

Lyndon words are known to have tight connections with the {\bf  shuffle algebra},
and indeed, the $\alp_{w,n}$ for arbitrary words $w$ of length $\leq n$ in $X$ satisfy natural {\bf shuffle relations} (Theorem \ref{shuffle relations}).
In \S\ref{section on n=2}-\S\ref{section on n=3} we show that for $n=2,3$ these shuffle relations fully describe $H^2(S^{[n,p]})$, provided that $p>2$, $p>3$, respectively (for $n=2$ this was essentially known).
Interestingly, related considerations arise also in the context of {\sl multiple zeta values}, see e.g.\ \cite{MinhPetitot00}, although we are not aware of a direct connection.

\smallskip

The Lyndon words on $X$ form a special instance of {\sl Hall sets}, which are well-known to have fundamental role in the structure theory of free groups and free Lie algebras (see \cite{Reutenauer93}, \cite{SerreLie}).
In addition, the Lyndon words have a \textbf{triangularity property} (see Proposition \ref{Lam(tau w)}(b)).
This property allows us to construct certain upper-triangular unipotent matrices
that express a (semi-)duality between the $\sig_w$ and  the cohomology elements $\alp_{w,n}$.

\smallskip

We now outline the proof that the $\alp_{w,n}$ form a linear basis of $H^2(S^{[n,p]})$.
For simplicity we assume for the moment that $X$ is finite.
To each Lyndon word $w$ of length $1\leq s\leq n$ one associates in a canonical way an element $\tau_w$ of the $s$-th term of the lower central series of $S$ (see \S4).
The cosets of the powers $\sig_w=\tau_w^{p^{n-s}}$ generate $S^{(n,p)}/S^{(n+1,p)}$ (Theorem \ref{generators}).
Using the special structure of the lower $p$-central filtration of $\dbU$ we define for any two Lyndon words $w,w'$ of length $\leq n$ a value $\langle w,w'\rangle_n\in\dbZ/p$ (see \S\ref{section on pairings}).
The triangularity property of Lyndon words implies that the matrix $\bigl(\langle w,w'\rangle_n\bigr)$ is unipotent upper-triangular, whence invertible.
Turning now to cohomology, we define a natural perfect pairing
\[
(\cdot,\cdot)_n\colon S^{(n,p)}/S^{(n+1,p)}\times H^2(S^{[n,p]})\to\dbZ/p.
\]
Cohomological computations show that, for Lyndon words $w,w'$ of length $\leq n$, one has $\langle w,w'\rangle_n=(\sig_w,\alp_{w',n})_n$.
Hence the matrix $\bigl((\sig_w,\alp_{w',n})_n\bigr)$ is also invertible.
We then conclude that the $\alp_{w',n}$ form a basis of $H^2(S^{[n,p]})$ (Theorem \ref{basis}).
This immediately determines the dimension of the latter cohomology group, in terms of Witt's
{\bf necklace function}, which counts the number of Lyndon words over $X$ of a given length (Corollary \ref{dim of H2}).

\smallskip

In the special case $n=2$, the theory developed here generalizes the above description of $H^2(S^{[2,p]})$ in terms of the Bockstein map and cup products (see \S\ref{section on n=2} for details).
Namely, the matrix $(\langle w,w'\rangle_2)$ is the identity matrix, which gives the above duality between $p$th powers/commutators and Bockstein elements/cup products.
Likewise, the shuffle relations just recover the basic fact that the cup product factors via the exterior product.

In \S\ref{section on n=3} we describe our theory explicitly also for the (new) case $n=3$.

\smallskip

While here we focus primarily on free profinite groups, it may be interesting to study the canonical elements $\bar\rho_*\gam_{n,s}$ for more general profinite groups $G$, in particular, when $G=G_F$ is the absolute Galois group of a field $F$.
For instance, when $n=2$, they were used in \cite{EfratMinac11} (following \cite{MinacSpira96} and \cite{AdemKaragueuzianMinac99}) and \cite{CheboluEfratMinac12} to describe the quotient $G_F^{[3,p]}$.
Triple Massey products for $G_F$ (which correspond to the case $n=s=3$) were also extensively studied in recent years -- see \cite{EfratMatzri15},  \cite{MinacTan15b}, \cite{MinacTan16}, \cite{MinacTan17}, and \cite{Wickelgren12} and the references therein.

I thank Claudio Quadrelli and the anonymous referee for their careful reading of this paper and for their very valuable comments and suggestions on improving the exposition.

\section{Words}
\label{section on words}
Let $X$ be a nonempty set, considered as an alphabet.
Let $X^*$ be the free monoid on $X$.
We view its elements as associative words on $X$.
The length of a word $w$ is denoted by $|w|$.
We write $\emptyset$ for the empty word, and $ww'$ for the concatenation of words $w$ and $w'$.

Recall that a \textbf{magma} is a set $\calM$ with a binary operation $(\cdot,\cdot)\colon \calM\times \calM\to \calM$.
A morphism of magmas is a map which commutes with the associated binary operations.
There is a \textbf{free magma} $\calM_X$ on $X$, unique up to an isomorphism;
that is, $X\subseteq \calM_X$, and for every magma $(\cdot,\cdot)\colon N\times N\to N$ and a map $f_0\colon X\to N$ there is a magma morphism $f\colon \calM_X\to N$ extending $f_0$.
See  \cite{SerreLie}*{Ch.\ IV, \S1} for an explicit construction of $\calM_X$.
The elements of $\calM_X$ may be viewed as non-associative words on $X$.

The monoid $X^*$ is a magma with respect to concatenation, so the universal property of $\calM_X$ gives rise to a unique magma morphism $f\colon \calM_X\to X^*$, called  the \textbf{foliage} (or \textbf{brackets dropping}) map, such that $f(x)=x$ for $x\in X$.

Let $\calH$ be a subset of $\calM_X$ and $\leq$ a total order on $\calH$.
We say that $(\calH,\leq)$ is a \textbf{Hall set in $\calM_X$}, if the following conditions hold \cite{Reutenauer93}*{\S4.1}:
\begin{enumerate}
\item[(i)]
$X\subseteq \calH$;
\item[(ii)]
If $h=(h',h'')\in\calH\setminus X$, then $h<h''$;
\item[(iii)]
For $h=(h',h'')\in \calM_X \setminus X$, one has $h\in\calH$ if and only if
\begin{itemize}
\item $h',h''\in\calH$ and $h'<h''$; and
\item either $h'\in X$, or $h'=(v,u)$ where $u\geq h''$.
\end{itemize}
\end{enumerate}
Given a Hall set $(\calH,\leq)$ in $\calM_X$ we call $H=f(\calH)$ a \textbf{Hall set in $X^*$}.

Every $w\in H$ can be written as $w=f(h)$ for a \textsl{unique} $h\in\calH$ \cite{Reutenauer93}*{Cor.\ 4.5}.
If $w\in H\setminus X$,  then we can uniquely write $h=(h',h'')$ with $h',h''\in \calH$, and call $w=w'w''$, where $w'=f(h')$ and $w''=f(h'')$, the \textbf{standard factorization} of $w$ \cite{Reutenauer93}*{p.\ 89}.

\medskip

Next we fix a total order $\leq$ on $X$, and define a total order $\leq_{\rm alp}$ (the {\bf alphabetical} order)  on $X^*$ as follows:
Let $w_1,w_2\in X^*$.
 Then $w_1\leq_{\rm alp} w_2$ if and only if $w_2=w_1v$ for some $v\in X^*$, or $w_1=vx_1u_1$, $w_2=vx_2u_2$ for some words $v,u_1,u_2$ and some letters $x_1,x_2\in X$ with $x_1<x_2$.
Note that the restriction of $\leq_{\rm alp}$ to $X^n$ is the lexicographic order.

In addition, we order $\dbZ_{\geq0}\times X^*$ lexicographically with respect to the usual order on $\dbZ_{\geq0}$ and the order $\leq_{\rm alp}$ on $X^*$.
We then define a second total order  $\preceq$ on $X^*$ by setting
\begin{equation}
\label{preceq}
w_1\preceq w_2\quad \Longleftrightarrow\quad (|w_1|,w_1)\leq(|w_2|,w_2)
\end{equation}
with respect to the latter order on $\dbZ_{\geq0}\times X^*$.

A nonempty word $w\in X^*$ is called a \textbf{Lyndon word} if  it is smaller in $\leq_{\rm alp}$ than all its non-trivial proper right factors.
Equivalently, no non-trivial rotation leaves $w$ invariant, and $w$ is lexicographically strictly smaller than all its rotations $\neq w$ (\cite{ChenFoxLyndon58}*{Th.\ 1.4}, \cite{Reutenauer93}*{Cor.\ 7.7}).
We denote the set of all Lyndon words on $X$ by $\Lyn(X)$, and the set of all such words of length $n$ (resp., $\leq n$) by $\Lyn_n(X)$ (resp., $\Lyn_{\leq n}(X)$).
The set $\Lyn(X)$, totally ordered with respect to $\leq_{\rm alp}$, is a Hall set \cite{Reutenauer93}*{Th.\ 5.1}.

\begin{exam}
\label{examples of Lyndon words}
\rm
The Lyndon words of length $\leq4$ on $X$ are
\[
\begin{split}
&(x) \hbox{ for } x\in X,\\
& (xy), (xxy), (xyy), (xxxy), (xxyy), (xyyy) \hbox{ for } x,y\in X, \ x<y, \\
&  (xyz), (xzy),(xxyz), (xxzy), (xyxz), (xyyz), (xyzy), (xyzz), \\
& \qquad\qquad \qquad(xzyy), (xzyz), (xzzy) \hbox{  for } x,y,z\in X, \  x<y<z,\\
& (xyzt), (xytz), (xzyt),(xzty),(xtyz), (xtzy) \hbox{  for } x,y,z,t\in X, \  x<y<z<t
\end{split}
\]
\end{exam}

The {\bf necklace map} (also called the {\bf Witt map}) is defined for integers $n,m\geq1$ by
\[
\varphi_n(m)=\frac1n\sum_{d|n}\mu(d)m^{n/d}.
\]
Here $\mu$ is the M\"obius function, defined by $\mu(d)=(-1)^k$, if $d$ is a product of $k$ distinct prime numbers, and $\mu(d)=0$ otherwise;
Alternatively,  $1/\zeta(s)=\sum_{n=1}^\infty\mu(n)/n^s$, where $\zeta(s)$ is the Riemann zeta function and $s$ is a complex number with real part $>1$.
When $m=q$ is a prime power, $\varphi_n(q)$ also counts the number of irreducible monic polynomials of degree $n$ over a field of $q$ elements \cite{Reutenauer93}*{\S7.6.2}.
We also define $\varphi_n(\infty)=\infty$.
One has \cite{Reutenauer93}*{Cor.\ 4.14}
\begin{equation}
\label{size of Lyn}
|\Lyn_n(X)|=\varphi_n(|X|).
\end{equation}

\section{Power series}
\label{section on power series}
We fix a commutative unital ring $R$.
Recall that a bilinear map $M\times N\to R$ of $R$-modules is {\bf non-degenerate} if its left and right kernels are trivial, i.e., the induced maps $M\to\Hom(N,R)$ and $N\to\Hom(M,R)$ are injective.
The bilinear map is {\bf perfect} if these two maps are isomorphisms.

Let $R\langle X\rangle$ be the free associative $R$-algebra on $X$.
We may view its elements as polynomials over $R$ in the non-commuting variables $x\in X$.
The additive group of $R\langle X\rangle$ is the free $R$-module on the basis $X^*$.
We grade $R\langle X\rangle$ by total degree.
Let $R\langle \langle X\rangle\rangle$ be the ring of all formal power series in the non-commuting variables $x\in X$ and coefficients in $R$.
For $f\in R\langle\langle X\rangle\rangle$ we write $f_w$ for the coefficient of $f$ at $w\in X^*$.
Thus $f=\sum_{w\in X^*}f_ww$.
There are natural embedings $X^*\subseteq R\langle X\rangle\subseteq R\langle\langle X\rangle\rangle$, where we identify $w\in X^*$ with the series $f$ such that $f_w=1$ and $f_{w'}=0$ for $w'\neq w$.
There is a well-defined non-degenerate bilinear map of  $R$-modules
\begin{equation}
\label{bilinear}
R\langle\langle X\rangle\rangle\times R\langle X\rangle\to R, \quad (f,g)\mapsto \sum_{w\in X^*}f_wg_w;
\end{equation}
See \cite{Reutenauer93}*{p.\ 17}.
For every integer $n\geq0$, it restricts to a non-degenerate bilinear map
on the homogenous components of degree $n$.

We may identify the additive group of $R\langle\langle X\rangle\rangle$ with $R^{X^*}$ via the map $f\mapsto(f_w)_w$.
When $R$ is equipped with a profinite ring topology (e.g.\ when $R$ is finite), the product topology on $R^{X^*}$  induces a profinite group topology on $R\langle\langle X\rangle\rangle$.
Moreover, the multiplication map of $R\langle\langle X\rangle\rangle$ is continuous, making it a profinite ring.
The group $R\langle\langle X\rangle\rangle^\times$ of all invertible elements in $R\langle\langle X\rangle\rangle$ is then a profinite group.

Next we recall from \cite{FriedJarden08}*{\S17.4} the following terminology and facts on free profinite groups.
Let $G$ be a profinite group and $X$ a set.
A map $\psi\colon X\to G$ {\bf converges to $1$}, if for every open normal subgroup $N$ of $G$, the set $X\setminus\psi\inv(N)$ is finite.
We say that a profinite group $S$ is a {\bf free profinite group on basis $X$} with respect to a map $\iota\colon X\to S$ if
\begin{enumerate}
\item[(i)]
$\iota\colon X\to S$ converges to $1$ and $\iota(X)$ generates $S$ as a profinite group;
\item[(ii)]
For every profinite group $G$ and a map $\psi\colon X\to G$ converging to $1$, there is a unique continuous homomorphism $\hat\psi\colon S\to G$ such that $\psi=\hat\psi\circ\iota$ on $X$.
\end{enumerate}
A free profinite group on $X$ exists, and is unique up to a continuous isomorphism.
We denote it by $S_X$.
The map $\iota$ is then injective, and we identify $X$ with its image in $S_X$.

We define the (profinite) {\bf Magnus homomorphism} $\Lam_{X,R}\colon S_X\to R\langle\langle  X\rangle\rangle^\times$ as follows (compare \cite{Efrat14}*{\S5}):

Assume first that $X$ is finite.
For  $x\in X$ one has $1=(1+x)\sum_{i=0}^\infty(-1)^ix^i$ in $R\langle\langle X\rangle\rangle$, so $1+x\in R\langle\langle X\rangle\rangle^\times$.
Hence, by (ii), the map $\psi\colon X\to R\langle\langle X\rangle\rangle^\times$, $x\mapsto 1+x$, uniquely extends to a continuous homomorphism $\Lam_{X,R}\colon S_X\to R\langle\langle X\rangle\rangle^\times$.

Now suppose that $X$ is arbitrary.
Let $Y$ range over all finite subsets of $X$.
The map $\psi\colon X\to S_Y$, which is the identity on $Y$ and $1$ on $X\setminus Y$, converges to $1$.
Hence it extends to a unique continuous group homomorphism $S_X\to S_Y$.
Also, there is a unique continuous $R$-algebra homomorphism $R\langle\langle X\rangle\rangle\to R\langle\langle Y\rangle\rangle$, which is the identity on $Y$ and $0$ on $X\setminus Y$.
Then
\[
S_X=\invlim S_Y, \quad R\langle\langle X\rangle\rangle=\invlim R\langle\langle Y\rangle\rangle, \quad
 R\langle\langle X\rangle\rangle^\times=\invlim R\langle\langle Y\rangle\rangle^\times.
\]
We define $\Lam_{X,R}=\invlim\Lam_{Y,R}$.
It is functorial in both $X$ and $R$ in the natural way.
Note that $\Lam_{X,R}(x)=1+x$ for $x\in X$.

In the sequel, $X$ will be fixed, so we abbreviate $S=S_X$ and $\Lam_R=\Lam_{X,R}$.

\medskip

For $\sig\in S$ and a word $w\in X^*$ we denote the coefficient of $w$ in $\Lam_R(\sig)$ by $\eps_{w,R}(\sig)$.
Thus
\[
\Lam_R(\sig)=\sum_{w\in X^*}\eps_{w,R}(\sig)w.
\]
By the construction of $\Lam_R$, we have $\eps_{\emptyset,R}(\sig)=1$.
Since $\Lam_R$ is a homomorphism,  for every $\sig,\tau\in S$ and $w\in X^*$ one has
\begin{equation}
\label{eps of product}
\eps_{w,R}(\sig\tau)=\sum_{w=u_1u_2}\eps_{u_1,R}(\sig)\eps_{u_2,R}(\tau).
\end{equation}

We will also need the classical {\sl discrete} version of the Magnus homomorphism.
To define it, assume that $X$ is finite, and let $F_X$ be the free group on basis $X$.
There is a natural homomorphism $F_X\to S_X$.
The {\bf discrete Magnus homomorphism} $\Lam_\dbZ^{\rm disc}\colon F_X\to\dbZ\langle\langle X\rangle\rangle^\times$ is defined again by $x\mapsto 1+x$.
There is a commutative square
\begin{equation}
\label{discrete Magnus hom}
\xymatrix{
F_X\ar[r] \ar[d]_{\Lam_\dbZ^{\rm disc}} & S_X\ar[d]^{\Lam_{\dbZ_p}}\\
\dbZ\langle\langle X\rangle\rangle^\times\ar@{^{(}->}[r] &\dbZ_p\langle\langle X\rangle\rangle^\times.
}
\end{equation}

\section{Lie algebra constructions}
\label{section on Lie algebra constructions}
Recall that the {\bf lower central filtration} $G^{(n,0)}$, $n=1,2\nek$ of a profinite group $G$ is defined inductively by
\[
G^{(1,0)}=G,\quad G^{(n+1,0)}=[G^{(n,0)},G].
\]
Thus $G^{(n+1,0)}$ is generated as a profinite group by all elements of the form  $[\sig,\tau]$ with $\sig\in G^{(n,0)}$ and $\tau\in G$.
One has $[G^{(n,0)},G^{(m,0)}]\leq G^{(n+m,0)}$ for every $n,m\geq1$ (compare \cite{SerreLie}*{Part I, Ch.\ II, \S3}).

\begin{prop}
\label{Koch}
Let $S=S_X$ and let $\sig\in S$.
Then:
\begin{enumerate}
\item[(a)]
$\sig\in S^{(n,0)}$ if and only if $\eps_{w,\dbZ_p}(\sig)=0$ for every $w\in X^*$ with $1\leq|w|<n$;
\item[(b)]
$\sig\in S^{(n,p)}$ if and only if $\eps_{w,\dbZ_p}(\sig)\in p^{n-|w|}\dbZ_p$ for every $w\in X^*$ with $1\leq|w|<n$.
\end{enumerate}
\end{prop}
\begin{proof}
In the discrete case (a) and (b) are due to Gr\"un and Magnus (see \cite{SerreLie}*{Part I, Ch.\ IV, th. 6.3}) and Koch \cite{Koch60}, respectively.
The results in the profinite case follow by continuity.

For other approaches see \cite{Morishita12}*{Prop.\ 8.15}, \cite{ChapmanEfrat16}*{Example 4.5}, and \cite{MinacTan15a}*{Lemma 2.2(d)}.
\end{proof}

We will need the following profinite analog of \cite{Fenn83}*{Lemma 4.4.1(iii)}.

\begin{lem}
\label{commutators}
Let $\sig\in S^{(n,0)}$ and $\tau\in S^{(m,0)}$, and let $w\in X^*$ have length $n+m$.
Write $w=u_1u_2=u'_2u'_1$ with $|u_1|=|u'_1|=n$ and $|u_2|=|u'_2|=m$.
Then
\[
\eps_{w,\dbZ_p}([\sig,\tau])=\eps_{u_1,\dbZ_p}(\sig)\eps_{u_2,\dbZ_p}(\tau)-\eps_{u'_2,\dbZ_p}(\tau)\eps_{u'_1,\dbZ_p}(\sig).
\]
\end{lem}
\begin{proof}
By Proposition \ref{Koch}(a), we may write $\Lam_{\dbZ_p}(\sig)=1+P+O(n+1)$ and $\Lam_{\dbZ_p}(\tau)=1+Q+O(m+1)$, where $P,Q\in\dbZ_p\langle\langle X\rangle\rangle$ are homogenous of degrees $n,m$, respectively, and where $O(r)$ denotes a power series containing only terms of degree $\geq r$.
Then
\[
\Lam_{\dbZ_p}([\sig,\tau])=1+PQ-QP+O(n+m+1).
\]
(compare e.g., \cite{Morishita12}*{Proof of Prop.\ 8.5}).
By (\ref{eps of product}), it follows that
\[
\begin{split}
\eps_{w,\dbZ_p}([\sig,\tau])&=(PQ-QP)_w=P_{u_1}Q_{u_2}-Q_{u'_2}P_{u'_1}\\
&=\eps_{u_1,\dbZ_p}(\sig)\eps_{u_2,\dbZ_p}(\tau)-\eps_{u'_2,\dbZ_p}(\tau)\eps_{u'_1,\dbZ_p}(\sig).
\end{split}
\]
\end{proof}

The commutator map induces on the graded ring $\bigoplus_{n=1}^\infty S^{(n,0)}/S^{(n+1,0)}$ a graded Lie algebra structure \cite{SerreLie}*{Part I, Ch.\ II, Prop.\ 2.3}.
Let $\grd$ be the ideal in the $\dbZ_p$-algebra $\dbZ_p\langle\langle X\rangle\rangle$ generated by $X$.
Then $\bigoplus_{n=1}^\infty\grd^n/\grd^{n+1}$ is a Lie algebra with the Lie brackets defined on homogenous components by $[\bar f,\bar g]=\overline{fg-gf}$ for $f\in\grd^n$, $g\in \grd^m$ \cite{SerreLie}*{p.\ 25}.
By Proposition \ref{Koch}(a), $\Lam_{\dbZ_p}$ induces a graded $\dbZ_p$-algebra homomorphism
\[
\gr\ \Lam_{\dbZ_p}\colon \bigoplus_{n=1}^\infty S^{(n,0)}/S^{(n+1,0)}\to\bigoplus_{n=1}^\infty\grd^n/\grd^{n+1},
\quad
\sig S^{(n+1,0)}\mapsto \sum_{|w|=n}\eps_{w,\dbZ_p}(\sig)+\grd^{n+1}.
\]
Then Lemma \ref{commutators} means that $\gr\ \Lam_{\dbZ_p}$ is a Lie algebra homomorphism.

For  $w\in \Lyn(X)$ we inductively define an element  $\tau_w$ of $S$ and  a non-commutative polynomial $P_w\in \dbZ\langle X\rangle\subseteq\dbZ_p\langle X\rangle$ as follows:
\begin{itemize}
\item
If $w=(x)$ has length $1$, then $\tau_w=x$ and $P_w=x$;
\item
If $|w|>1$, then we take the standard factorization $w=w'w''$ of $w$ with respect to the Hall set $\Lyn(X)$ (see \S\ref{section on words}), and set
\[
\tau_w=[\tau_{w'},\tau_{w''}], \qquad P_w=P_{w'}P_{w''}-P_{w''}P_{w'}.
\]
\end{itemize}

For $w\in\Lyn_n(X)$ one has $\tau_w\in S^{(n,0)}$.
Moreover:

\begin{prop}
\label{the taus generate}
Let $n\geq1$.
The cosets of $\tau_w$, with $w\in \Lyn_n(X)$, generate $S^{(n,0)}/S^{(n+1,0)}$.
\end{prop}
\begin{proof}
See \cite{Reutenauer93}*{Cor.\ 6.16} for the discrete version.
The profinite version follows by taking closure.
\end{proof}

Let $\leq_{\rm alp}$ and $\preceq$ be the total orders on $X^*$ defined in \S\ref{section on words}.
The importance of the Lyndon words in our context, beside forming a Hall set, is part (b)
 of the following Proposition, called the {\bf triangularity} property.

\begin{prop}
\label{Lam(tau w)}
Let $w\in\Lyn(X)$.
Then
\begin{enumerate}
\item[(a)]
$\Lam_{\dbZ_p}(\tau_w)-1-P_w$ is a  combination of words of length $>|w|$.
\item[(b)]
$P_w-w$ is  a combination of words of length $|w|$  which are strictly larger than $w$ with respect to $\leq_{\rm alp}$.
\item[(c)]
$\Lam_{\dbZ_p}(\tau_w)-1-w$ is a combination of words which are strictly larger than $w$ in $\preceq$.
\end{enumerate}
\end{prop}
\begin{proof}
(a) \quad
Since $\gr\,\Lam_{\dbZ_p}$ is a Lie algebra homomorphism, for $w\in\Lyn_n(X)$ we have by induction
\[
(\gr\,\Lam_{\dbZ_p})(\tau_w S^{(n+1,0)})=P_w+\grd^{n+1},
\]
and the assertion follows.
See also \cite{Reutenauer93}*{Lemma 6.10(ii)}.

\medskip

(b) \quad
See \cite{Reutenauer93}*{Th.\ 5.1} and its proof.

\medskip

(c) \quad
This follows from (a) and (b).
\end{proof}

\section{Generators for $S^{(n,p)}/S^{(n+1,p)}$}
Let $\pi$ be an indeterminate over the ring $\dbZ/p$ and let $\dbZ/p[\pi]$ be the polynomial ring.
Let $A_\bullet=\bigoplus_{n=1}^\infty A_n$ be a graded Lie $\dbZ/p$-algebra with Lie bracket $[\cdot,\cdot]$.
Suppose that there is a map
$\dbZ/p[\pi]\times A_\bullet\to A_\bullet, \quad (\alp,\xi)\mapsto \alp\xi$,
which is $\dbZ/p$-linear in $\dbZ/p[\pi]$, such that $\pi\xi\in A_{s+1}$ for $\xi\in A_s$,  and such that for every $\xi_1,\xi_2\in A_s$,
\begin{equation}
\label{properties of pi}
\pi(\xi_1+\xi_2)=\begin{cases}
\pi\xi_1+\pi\xi_2, & \hbox{ if } p>2 \hbox{ or } s>1,\\
\pi\xi_1+\pi\xi_2+[\xi_1,\xi_2], & \hbox{ if }p=2,\ s=1.
\end{cases}
\end{equation}
By induction, this extends to:

\begin{lem}
\label{ZpPi algebras}
Let $r,k,s\geq1$  and let $\xi_1\nek\xi_k\in A_s$.
Then
\[
\pi^r(\sum_{i=1}^k\xi_i)=
\begin{cases}
\sum_{i=1}^k\pi^r\xi_i,& \hbox{ if }p>2 \hbox{ or } s>1,\\
\sum_{i=1}^k\pi^r\xi_i+\sum_{i<j}\pi^{r-1}[\xi_i,\xi_j],&\hbox{ if }p=2,\ s=1.
\end{cases}
\]
\end{lem}

We write $\langle T\rangle$ for the submodule of $A_\bullet$ generated by a subset $T$.

\begin{lem}
\label{cor on gens}
Let $n\geq2$ and for each $1\leq s\leq n$ let $T_s$ be a subset of $A_s$.
When $p=2$ assume also that $[\tau_1,\tau_2]\in T_2\cup\{0\}$  for every $\tau_1,\tau_2\in T_1$.
If the sets $\pi^{n-s}\langle T_s\rangle$,  $s=1,2\nek n$, generate $A_n$, then the sets $\pi^{n-s}T_s$, $s=1,2\nek n$, also generate $A_n$.
\end{lem}
\begin{proof}
When $p>2$ or $s>1$ Lemma \ref{ZpPi algebras} shows that $\pi^{n-s}\langle T_s\rangle=\langle\pi^{n-s}T_s\rangle$.
When $p=2$ and $s=1$, it shows that
\[
\pi^{n-1}\langle T_1\rangle\subseteq \langle\pi^{n-1}T_1\rangle+\langle\pi^{n-2}T_2\rangle\subseteq A_n.
\]
Therefore the subgroup of $A_n$ generated by the sets $\pi^{n-s}T_s$, $s=1,2\nek n$, contains the sets $\pi^{n-s}\langle T_s\rangle$, $s=1,2\nek n$, and hence equals $A_n$.
\end{proof}

Motivated by e.g., \cite{Lazard54}, \cite{SerreDemuskin},  \cite{Labute67}, we now specialize to a graded Lie algebra defined using the lower $p$-central filtration.
We refer to \cite{NeukirchSchmidtWingberg}*{Ch.\ III, \S8} for the following facts.
For the free profinite group $S$ on the basis $X$ and for $n\geq1$ we set $\gr_n(S)=S^{(n,p)}/S^{(n+1,p)}$.
It is an elementary abelian $p$-group, which we write additively.
The commutator map induces a map $[\cdot,\cdot]\colon \gr_n(S)\times\gr_m(S)\to\gr_{n+m}(S)$, which
endows a graded Lie algebra structure on $\gr_\bullet(S)=\bigoplus_{n=1}^\infty\gr_n(S)$.
The map $\tau\mapsto \tau^p$  maps $S^{(r,p)}$ into $S^{(r+1,p)}$, and induces  a map $\pi_r\colon\gr_r(S)\to\gr_{r+1}(S)$.
The map $(\pi^r,\xi)\mapsto \pi_r(\xi)$ for $\xi\in\gr_r(S)$ extends to a map
$\dbZ/p[\pi]\times\gr_\bullet(S)\to\gr_\bullet(S)$ which is $\dbZ/p$-linear in the first component and which satisfies (\ref{properties of pi}).

\begin{thm}
\label{generators}
Let $n\geq1$.
The cosets of $\tau_w^{p^{n-s}}$, with $1\leq s\leq n$ and $w\in \Lyn_s(X)$, generate $S^{(n,p)}/S^{(n+1,p)}$.
\end{thm}
\begin{proof}
The case $n=1$ is immediate, so we assume that $n\geq2$.
For every $1\leq s\leq n$ let $T_s$ be the set of cosets of $\tau_w$, $w\in\Lyn_s(X)$, in $\gr_s(S)$.
By Proposition \ref{the taus generate}, $\langle T_s\rangle$ is the image of $S^{(s,0)}$ in $\gr_s(S)$.
Hence $\pi^{n-s}\langle T_s\rangle$ consists of the cosets in $\gr_n(S)$ of the $p^{n-s}$-powers of $S^{(s,0)}$.
One has
\[
S^{(n,p)}=\prod_{s=1}^n(S^{(s,0)})^{p^{n-s}}
\]
\cite{NeukirchSchmidtWingberg}*{Prop.\ 3.8.6}.
Thus the sets $\pi^{n-s}\langle T_s\rangle$, $s=1,2\nek n$, generate $\gr_n(S)$.

Further, $T_1$ consists of the cosets of $x$, with $x\in X$, and $T_2$ consists of the cosets of commutators $[x,y]$ with $x<y$ in $X$.
Moreover, when $p=2$  the cosets of $[x,y]$ and $[y,x]=[x,y]\inv$ in $\gr_2(S)$ coincide.
Lemma \ref{cor on gens} therefore implies that even the sets $\pi^{n-s}T_s$, $s=1,2\nek n$, generate $\gr_n(S)$, as required.
\end{proof}

\section{The pairing $\langle w,w'\rangle_n$}
\label{section on pairings}
For a commutative unitary ring $R$ and a positive integer $m$, let $\dbU_m(R)$ be the group of all $m\times m$ upper-triangular unipotent matrices over $R$.
We write $I_m$ for the identity matrix in $\dbU_m(R)$, and $E_{ij}$ for the matrix with $1$ at entry $(i,j)$ and $0$ elsewhere.
As above, $X$ will be a totally ordered set.

For the following fact we refer, e.g., to \cite{Efrat14}*{Lemma 7.5}.
We recall from \S\ref{section on power series} that $\eps_{u,R}(\sig)$ is the coefficient of the word $u\in X^*$ in  the formal power series $\Lam_R(\sig)\in R\langle\langle X\rangle\rangle^\times$.

\begin{prop}
\label{rho}
Given a profinite ring $R$ and a word $w=(x_1\cdots x_s)$ in $X^*$ there is a well defined homomorphism of profinite groups
\[
\rho^w_R\colon S\to \dbU_{s+1}(R), \quad \sig\mapsto (\eps_{(x_i\cdots x_{j-1}),R}(\sig))_{1\leq i<j\leq s+1}.
\]
\end{prop}

\begin{rem}
\label{dual basis}
\rm
In particular, for each $x\in X$ the map $\chi_{x,R}=\eps_{(x),R}\colon S\to R$ is a group homomorphism, where $R$ is considered as an additive group.
The homomorphisms $\chi_{x,R}$, $x\in X$, are dual to the basis $X$, in the sense that $\chi_{x,R}(x)=1$, and $\chi_{x,R}(y)=0$ for $x\neq y$ in $X$.
\end{rem}

\begin{prop}
\label{properties of dbU}
Let  $1\leq s\leq n$.
For $R=\dbZ/p^{n-s+1}$ one has:
\begin{enumerate}
\item[(a)]
$\dbU_{s+1}(R)^{(n,p)}=I_{s+1}+\dbZ p^{n-s}E_{1,s+1}$.
\item[(b)]
$\dbU_{s+1}(R)^{(n,p)}$ is central in $\dbU_{s+1}(R)$.
\end{enumerate}
\end{prop}
\begin{proof}
(a) \quad
We follow the argument of \cite{MinacTan15a}*{Lemma 2.4}.
Take $X=\{x_1\nek x_s\}$ be a set of $s$ elements, let $S=S_X$, and let $w=(x_1\cdots x_s)$.
The matrices $\rho^w_R(x_i)=I_{s+1}+E_{i,i+1}$, $i=1,2\nek s$, generate $\dbU_{s+1}(R)$ \cite{Weir55}*{p.\ 55}, so $\rho^w_R$ is surjective.
Therefore it maps $S^{(n,p)}$ onto $\dbU_{s+1}(R)^{(n,p)}$.

By Proposition \ref{Koch}(b), for $\sig\in S^{(n,p)}$ and $u\in X^*$ of length $1\leq |u|\leq s$ one has $\eps_{u,\dbZ_p}(\sig)\in p^{n-|u|}\dbZ_p$.

If $|u|<s$, then $\eps_{u,\dbZ_p}(\sig)\in p^{n-|u|}\dbZ_p\subseteq p^{n-s+1}\dbZ_p$.
 Hence $\eps_{u,R}(\sig)=0$ in this case.

If $|u|=s$, then $\eps_{u,\dbZ_p}(\sig)\in p^{n-s}\dbZ_p$, so  $\eps_{u,R}(\sig)\in p^{n-s}R$.

Moreover, $\tau_w^{p^{n-s}}\in(S^{(s,0)})^{p^{n-s}}\leq S^{(n,p)}$.
By Proposition \ref{Lam(tau w)}(c),  $\Lam_{\dbZ_p}(\tau_w)=1+w+f$, where $f$ is a combination of words strictly larger than $w$ in $\preceq$.
Hence $\Lam_{\dbZ_p}(\tau_w^{p^{n-s}})=1+p^{n-s}w+g$, where $g$ is also a combination of words strictly larger than $w$ in $\preceq$, which implies that $\eps_{w,R}(\tau_w^{p^{n-s}})=p^{n-s}\cdot 1_R$.

Consequently, $\dbU_{s+1}(R)^{(n,p)}=\rho^w_R(S^{(n,p)})=I_{s+1}+\dbZ p^{n-s}E_{1,s+1}$.

\medskip

(b) \quad
It is straightforward to see that $I_{s+1}+\dbZ E_{1,s+1}$ is central in $\dbU_{s+1}(R)$, so the assertion follows from (a).
\end{proof}

See \cite{Borge04}*{\S2} for a related analysis of the lower $p$-central filtration of $\dbU_{s+1}(\dbZ/p^{n-s+1})$.

\smallskip

Consider the obvious isomorphism
\[
\iota_{n,s}\colon p^{n-s}\dbZ/p^{n-s+1}\dbZ\xrightarrow{\sim}\dbZ/p, \qquad ap^{n-s}\hbox{ (mod }p^{n-s+1}\hbox{)}\mapsto a\hbox{ (mod }p\hbox{)}.
\]
In view of Proposition \ref{properties of dbU}(a), we may define a group isomorphism
\[
\iota_{n,s}^\dbU\colon \dbU_{s+1}(\dbZ/p^{n-s+1})^{(n,p)}\xrightarrow{\sim}\dbZ/p, \quad
(a_{ij})\mapsto \iota_{n,s}(a_{1,s+1}).
\]

Next let $w,w'\in X^*$ be words of lengths $1\leq s,s'\leq n$, respectively, where $w$ is Lyndon.
We have $\tau_w^{p^{n-s}}\in (S^{(s,0)})^{p^{n-s}}\leq S^{(n,p)}$.
By Proposition \ref{Koch}(b),
$\eps_{w',\dbZ_p}(\tau_w^{p^{n-s}})\in p^{n-s'}\dbZ_p$, and therefore $\eps_{w',\dbZ/p^{n-s'+1}}(\tau_w^{p^{n-s}})\in p^{n-s'}\dbZ/p^{n-s'+1}\dbZ$.
We set
\begin{equation}
\label{description of pairing}
\langle w,w'\rangle_n=\iota_{n,s'}(\eps_{w',\dbZ/p^{n-s'+1}}(\tau_w^{p^{n-s}}))\in\dbZ/p.
\end{equation}
Alternatively,
\begin{equation}
\label{description of pairing 2}
\langle w,w'\rangle_n=\iota_{n,s'}^\dbU(\rho^{w'}_{\dbZ/p^{n-s'+1}}(\tau_w^{p^{n-s}})).
\end{equation}

Let $\preceq$ be as in (\ref{preceq}).

\begin{prop}
\label{vanishing of iota}
Let $w,w'$ be words in $X^*$ of lengths $1\leq s,s'\leq n$, respectively, with $w$ Lyndon.
\begin{enumerate}
\item[(a)]
If $w'\prec w$, then  $\langle w,w'\rangle_n=0$;
\item[(b)]
$\langle w,w\rangle_n=1$;
\item[(c)]
If $w'$ contains letters which do not appear in $w$, then  $\langle w,w'\rangle_n=0$;
\item[(d)]
If $s<s'<2s$, then  $\langle w,w'\rangle_n=0$.
\end{enumerate}
\end{prop}
\begin{proof}
(a), (b) \quad
Proposition \ref{Lam(tau w)}(c) implies that $\Lam_{\dbZ_p}(\tau_w^{p^{n-s}})-1-p^{n-s}w$ is a combination of words strictly larger than $w$ with respect to $\preceq$,
and the same therefore holds over the coefficient  ring $\dbZ/p^{n-s'+1}$.
Hence, if $w'\prec w$, then $\eps_{w',\dbZ/p^{n-s'+1}}(\tau_w^{p^{n-s}})=0$, so $\langle w,w'\rangle_n=0$.
If $w=w'$, then $\eps_{w,\dbZ/p^{n-s+1}}(\tau_w^{p^{n-s}})=p^{n-s}\cdot1_{\dbZ/p^{n-s+1}}$, whence  $\langle w,w\rangle_n=1$.

\medskip

(c) \quad
Here we clearly have $\eps_{w',\dbZ/p^{n-s'+1}}(\tau_w^{p^{n-s}})=0$.

\medskip

(d) \quad
Since $\tau_w\in S^{(s,0)}$, one may write $\Lam_{\dbZ_p}(\tau_w)=1+P+O(s'+1)$, where $P$ is a combination of words $w''$ of length $s\leq |w''|\leq s'$, and $O(s'+1)$ denotes a combination of words of length $\geq s'+1$ (Proposition \ref{Koch}(a)).
Since $s'<2s$, this implies that $\Lam_{\dbZ_p}(\tau_w^{p^{n-s}})=1+p^{n-s}P+O(s'+1)$.
In particular, $\eps_{w',\dbZ_p}(\tau_w^{p^{n-s}})\in p^{n-s}\dbZ_p$,
and therefore
\[
\eps_{w',\dbZ/p^{n-s'+1}}(\tau_w^{p^{n-s}})\in p^{n-s}(\dbZ/p^{n-s'+1})=\{0\},
\]
since  $s<s'$.
Hence  $\langle w,w'\rangle_n=0$.
\end{proof}

\section{Transgressions}
\label{section on transgression}
Given a profinite group $G$ and a discrete $G$-module $A$, we write as usual $C^i(G,A)$, $Z^i(G,A)$, and $H^i(G,A)$ for the corresponding group of continuous $i$-cochains, group of continuous $i$-cocycles, and the $i$th profinite cohomology group, respectively.
For $x\in Z^i(G,A)$ let $[x]$ be its cohomology class in $H^i(G,A)$.

For a normal closed subgroup $N$ of $G$, let $\trg\colon H^1(N,A)^G\to H^2(G/N,A^N)$ be the transgression homomorphism.
It is the map $d_2^{0,1}$ of the Lyndon--Hochschild--Serre spectral sequence associated with $G$ and $N$  \cite{NeukirchSchmidtWingberg}*{Th.\ 2.4.3}.
We recall the explicit description of $\trg$, assuming for simplicity that the $G$-action on $A$ is trivial \cite{NeukirchSchmidtWingberg}*{Prop.\ 1.6.6}:
If $x\in Z^1(N,A)$, then there exists $y\in C^1(G,A)$ such that $y|_N=x$ and $(\partial y)(\sig_1,\sig_2)$ depends only on the cosets of $\sig_1,\sig_2$ modulo $N$, so that $\partial y$ may be viewed as an element of $Z^2(G/N,A)$.
For any such $y$ one has $\trg([x])=[\partial y]$.

We fix for the rest of this section a finite group $\dbU$ and a normal subgroup $N$ of $\dbU$ satisfying:
\begin{enumerate}
\item[(i)]
$N\isom \dbZ/p$; and
\item[(ii)]
$N$ lies in the center of $\dbU$.
\end{enumerate}
We set $\bar\dbU=\dbU/N$, and let it act trivially on $\dbU$.
We denote the image of $u\in\dbU$ in $\bar\dbU$ by $\bar u$.
We may choose a section $\lam$ of the projection $\dbU\to\bar\dbU$ such that $\lam(\bar 1)=1$.
We define a map $\del\in C^2(\bar\dbU,N)$ by
\[
\del(\bar u,\bar u')=\lam(\bar u)\cdot\lam(\bar u')\cdot\lam(\bar u\bar u')\inv.
\]
It is normalized, i.e., $\del(\bar u,1)=\del(1,\bar u)=1$ for every $\bar u\in\bar\dbU$.

We also define  $y\in C^1(\dbU,N)$ by $y(u)=u\lam(\bar u)\inv$.
Note that $y|_N=\id_N$.

\begin{lem}
\label{lemma with centers}
For every $u,u'\in\dbU$ one has
\[
\del(\bar u,\bar u')\cdot y(u)\cdot y(u')=y(uu').
\]
\end{lem}
\begin{proof}
Since $y(u)$ and $y(u')$ are in $N$, they are central in $\dbU$, so
\[
\begin{split}
\del(\bar u,\bar u')\cdot y(u)\cdot y(u')
&=\lam(\bar u)\cdot\lam(\bar u')\cdot\lam(\bar u\bar u')\inv  \cdot y(u)\cdot y(u')\\
&=y(u)\cdot\lam(\bar u)\cdot y(u')\cdot \lam(\bar u')\cdot\lam(\bar u\bar u')\inv \\
&=uu'\lam(\bar u\bar u')\inv=y(uu').
\end{split}
\]
\end{proof}

For the correspondence between elements of $H^2$ and central extensions see e.g., \cite{NeukirchSchmidtWingberg}*{Th.\ 1.2.4}.

\begin{prop}
\label{properties of del}
Using the notation above, the following holds.
\begin{enumerate}
\item[(a)]
$\del\in Z^2(\bar\dbU,N)$;
\item[(b)]
One has $\trg(\id_N)=-[\del]$ for the transgression map $\trg\colon H^1(N,N)^\dbU\to H^2(\bar\dbU,N)$.
\item[(c)]
The cohomology class $[\del]\in H^2(\bar\dbU,N)$ corresponds to the equivalence class of the central extension
\begin{equation}
\label{basic extension}
1\to N\to \dbU\to\bar\dbU\to1.
\end{equation}
\end{enumerate}
\end{prop}
\begin{proof}
(a), (b): \quad
For $u,u'\in\dbU$ Lemma \ref{lemma with centers} gives
\[
(\partial y)(u,u')=y(u)\cdot y(u')\cdot y(uu')\inv
=\del(\bar u,\bar u')\inv.
\]
This shows that $\del$ is a 2-cocycle, and that $(\partial y)(u,u')$ depends only on the cosets $\bar u,\bar u'$.
Further, $\id_N\in Z^1(N,N)$.
By the explicit description of the transgression above, $\trg(\id_N)=-[\del]$.

\medskip

(c) \quad
Consider the set $B=N\times\bar\dbU$ with the binary operation
\[
(u,\bar v)*(u',\bar v')=(\del(\bar v,\bar v')uu',\bar v\bar v').
\]
 The proof of \cite{NeukirchSchmidtWingberg}*{Th.\ 1.2.4} shows that  this makes $B$ a group, and $[\del]$ corresponds to the equivalence class of the central extension
\begin{equation}
\label{NBU}
1\to N\to B\to\bar\dbU\to1.
\end{equation}
Moreover,  the map $h\colon \dbU\to B$,  $u\mapsto(y(u),\bar u)$ is clearly bijective.
We claim that it is a homomorphism, whence an isomorphism.
Indeed, for $u,u'\in\dbU$ Lemma \ref{lemma with centers} gives:
\[
\begin{split}
h(u)*h(u')&=(y(u),\bar u) *(y(u'),\bar u')=(\del(\bar u,\bar u')y(u)y(u'),\bar u\bar u')\\
&=(y(uu'),\bar u\bar u')=h(uu').
\end{split}
\]
We obtain that the central extension (\ref{NBU}) is equivalent to the central extension (\ref{basic extension}).
\end{proof}

Next let $\bar G$ be a profinite group, and let $\bar\rho\colon \bar G\to\bar\dbU$ be a continuous homomorphism.
Let $\iota\colon N\xrightarrow{\sim}\dbZ/p$ be a fixed isomorphism (see (i)).
Set
\begin{equation}
\label{alp}
\alp=(\bar\rho^*\circ\iota_*)([\del])=-(\bar\rho^*\circ\iota_*\circ\trg)(\id_N)\in H^2(\bar G,\dbZ/p),
\end{equation}
where the second equality is by Proposition \ref{properties of del}(b).
Then $\alp$ corresponds to the equivalence class of the central extension
\begin{equation}
\label{central extension revised}
0\to\dbZ/p\xrightarrow{\iota\inv\times1} \dbU\times_{\bar\dbU}\bar G\to \bar G\to1,
\end{equation}
where $\dbU\times_{\bar\dbU}\bar G$ is the fiber product with respect to the natural projection $\dbU\to\bar\dbU$ and to $\bar\rho$;
See \cite{Hoechsmann68}*{Proof of 1.1}.

Suppose further that there is a profinite group $G$, a closed normal subgroup $M$ of $G$, and a continuous homomorphism $\rho\colon G\to\dbU$ such that $\bar G=G/M$, $\rho(M)\leq N$, and $\bar\rho\colon \bar G\to\bar\dbU$ is induced from $\rho$.
The functoriality of transgression yields a commutative diagram
\[
\xymatrix{
H^1(N,N)^{\dbU}\ar[r]^{\iota_*}_{\sim}\ar[d]^{\trg}& H^1(N,\dbZ/p)^{\dbU}\ar[r]^{\rho^*}\ar[d]^{\trg} &
H^1(M,\dbZ/p)^G\ar[d]^{\trg} \\
H^2(\bar\dbU,N)\ar[r]^{\iota_*}_{\sim} & H^2(\bar\dbU,\dbZ/p)\ar[r]^{\bar\rho^*}& H^2(\bar G,\dbZ/p).}
\]
The image  of $\id_N\in H^1(N,N)$ in $H^1(M,\dbZ/p)^G$  is
\begin{equation}
\label{theta and rho}
\theta=\iota\circ(\rho|_M)\in H^1(M,\dbZ/p).
\end{equation}
By (\ref{alp}) and the commutativity of the diagram,
\begin{equation}
\label{alp and del}
\alp=-\trg(\theta)\in H^2(\bar G,\dbZ/p)
\end{equation}

\begin{rem}
\label{alp as a connecting hom}
\rm
Suppose that $\bar\dbU$ is abelian and that $\bar G$ acts trivially on $\bar\dbU$.
A $2$-cocycle representing $\alp$ is
\[
(\bar\sig,\bar\sig')\mapsto\iota(\lam(\bar\rho(\bar\sig))\cdot\lam(\bar\rho(\bar\sig'))\cdot\lam(\bar\rho(\bar\sig\bar\sig'))\inv).
\]
But $\lam(\bar\rho(\bar\sig))\cdot\lam(\bar\rho(\bar\sig'))\cdot\lam(\bar\rho(\bar\sig\bar\sig'))\inv$ is a $2$-cocycle representing the image of $\bar\rho$ under the connecting homomorphism $H^1(\bar G,\bar\dbU)\to H^2(\bar G,N)$ arising from (\ref{basic extension}).
Thus $\alp$ is the image of $\bar\rho$ under the composition
\[
H^1(\bar G,\bar\dbU)\to H^2(\bar G,N)\xrightarrow{\iota_*}H^2(\bar G,\dbZ/p).
\]
\end{rem}

\begin{exam}
\label{examples of  alp}
\rm
We give several examples of the above construction with the group $\dbU=\dbU_{s+1}(\dbZ/p^{n-s+1})$ (where $1\leq s\leq n$), a continuous homomorphism $\rho\colon G\to\dbU$, where $G$ is a profinite group, and the induced homomorphism $\bar\rho\colon  G^{[n,p]}\to \dbU^{[n,p]}$.
Note that assumptions (i) and (ii)  then hold for $N=\dbU^{(n,p)}$, by Proposition \ref{properties of dbU}.
We will be especially interested in the case where $G=S=S_X$ is a free profinite group, $M=S^{(n,p)}$, $\rho=\rho_{\dbZ/p^{n-s+1}}^{w}$ for a word $w\in X^*$ of length $1\leq s\leq n$, and
$\bar\rho=\bar\rho^w_{\dbZ/p^{n-s+1}}\colon S^{[n,p]}\to\dbU^{[n,p]}$ is the induced homomorphism.
In this setup we write $\alp_{w,n}$ for $\alp$.

\medskip

(1) {\sl Bocksteins.} \quad
For a positive integer $m$ and a profinite group $\bar G$, the connecting homomorphism arising from the short exact sequence of trivial $\bar G$-modules
\[
0\to \dbZ/p\to\dbZ/pm\to\dbZ/m\to0
\]
is the {\bf Bockstein homomorphism}
\[
\Bock_{m,\bar G}\colon H^1(\bar G,\dbZ/m)\to H^2(\bar G,\dbZ/p).
\]
Let  $\dbU=\dbU_2(\dbZ/p^n)$ (i.e., $s=1$).
There is a commutative diagram of central extensions
\[
\xymatrix{
1\ar[r]&\dbU^{(n,p)}\ar[r]\ar[d]_{\wr}&\dbU\ar[r]\ar[d]_{\wr}&\dbU^{[n,p]}\ar[r]\ar[d]_{\wr}&1\\
0\ar[r]&p^{n-1}\dbZ/p^n\dbZ\ar[r]&\dbZ/p^n\ar[r]&\dbZ/p^{n-1}\ar[r]&0.
}
\]
For a profinite group $\bar G$ and a homomorphism $\bar\rho\colon \bar G\to\dbZ/p^{n-1}$, Remark \ref{alp as a connecting hom} therefore implies that $\alp=\Bock_{p^{n-1},\bar G}(\bar\rho)$.

In particular, for $x\in X$ take
\[
\bar\rho=\bar\rho^{(x)}_{\dbZ/p^n}\colon \bar G=S^{[n,p]}\to \dbU_2(\dbZ/p^n)^{[n,p]}.
\]
Identifying $\dbU_2(\dbZ/p^n)^{[n,p]}=\dbZ/p^{n-1}$, we obtain
\[
\alp_{(x),n}=\Bock_{p^{n-1},S^{[n,p]}}(\eps_{(x),\dbZ/p^{n-1}}).
\]

\medskip

(2) \quad {\sl Massey products.} \quad
Let $n=s\geq2$, so $\dbU=\dbU_{n+1}(\dbZ/p)$.
Let  $\bar G$ be a profinite group, let $\bar\rho\colon \bar G\to\dbU^{[n,p]}$ be a continuous homomorphism, and let $\rho_{i,i+1}\colon \bar G\to\dbZ/p$ denote its projection on the $(i,i+1)$-entry, $i=1,2\nek n$.
By a result of Dwyer \cite{Dwyer75}*{Th.\ 2.6}, the extension (\ref{central extension revised}) then corresponds to a defining system for the $n$-fold Massey product
\[
\langle\rho_{12},\rho_{23}\nek\rho_{n,n+1}\rangle\subseteq H^2(\bar G,\dbZ/p)
 \]
(see \cite{Efrat14}*{Prop.\ 8.3} for the profinite analog of this fact).
Thus, for a fixed $\bar G$ and for homomorphisms $\bar\rho_1\nek\bar\rho_n\colon \bar G\to\dbZ/p$, with $\bar\rho$ varying over all homomorphisms such that $\bar\rho_{i,i+1}=\bar\rho_i$, $i=1,2\nek n$, the cohomology element $\alp$ ranges over the elements of the Massey product $\langle\bar\rho_1\nek\bar\rho_n\rangle$.

In particular, for $\bar G=S^{[n,p]}$ and for  a word $w=(x_1\cdots x_n)\in X^*$ of length $n$, the cohomology elements $\alp_{w,n}$ range over the Massey product $\langle\eps_{(x_1),\dbZ/p}\nek\eps_{(x_n),\dbZ/p}\rangle\subseteq H^2(S^{[n,p]},\dbZ/p)$, where the $\eps_{(x_i),\dbZ/p}$ are viewed as elements of $H^1(S^{[n,p]},\dbZ/p)$.

\medskip

(3) \quad {\sl Cup products.} \quad
In the special case $n=s=2$, the Massey product contains only the cup product.
Hence for every profinite group $\bar G$ and a homomorphism $\bar\rho\colon\bar G\to\dbU_3(\dbZ/p)$ the cohomology element $\alp\in H^2(\bar G,\dbZ/p)$ is the cup product $\bar\rho_{12}\cup\bar\rho_{23}$.
In particular,  for $w=(xy)$ we have
\[
\alp_{(xy),\dbZ/p}=\eps_{(x),\dbZ/p}\cup\eps_{(y),\dbZ/p}\in H^2(S^{[2,p]},\dbZ/p).
\]
\end{exam}

\section{Cohomological duality}
Let $S=S_X$ be again a free profinite group on the totally ordered set $X$, and let it act trivially on $\dbZ/p$.
Let $n\geq2$, so $S^{(n,p)}\leq S^p[S,S]$.
Then the inflation map $H^1(S^{[n,p]},\dbZ/p)\to H^1(S,\dbZ/p)$ is an isomorphism.
Further, $H^2(S,\dbZ/p)=0$.
By the five-term sequence of cohomology groups \cite{NeukirchSchmidtWingberg}*{Prop.\ 1.6.7}, $\trg\colon H^1(S^{(n,p)},\dbZ/p)^S\to H^2(S^{[n,p]},\dbZ/p)$ is an isomorphism.

There is a natural non-degenerate bilinear map
\[
S^{(n,p)}/S^{(n+1,p)}\times H^1(S^{(n,p)},\dbZ/p)^S\to\dbZ/p, \quad (\bar\sig,\varphi)\mapsto\varphi(\sig)
\]
(see \cite{EfratMinac11}*{Cor.\ 2.2}).
It induces a bilinear map
\[
(\cdot,\cdot)_n\colon S^{(n,p)}\times H^2(S^{[n,p]},\dbZ/p)\to\dbZ/p, \quad (\sig,\alp)_n=-(\trg\inv(\alp))(\sig),
\]
with left kernel $S^{(n+1,p)}$ and trivial right kernel.

Now let $w\in X^*$ be a word of length $1\leq s\leq n$.
As in Examples \ref{examples of  alp}, we apply the computations in Section \ref{section on transgression} to the group $\dbU=\dbU_{s+1}(\dbZ/p^{n-s+1})$, the open normal subgroup $N=\dbU^{(n,p)}$, the homomorphism $\rho=\rho^w_{\dbZ/p^{n-s+1}}\colon S\to\dbU$,  the induced homomorphism $\bar\rho=\bar\rho^w_{\dbZ/p^{n-s+1}}\colon S^{[n,p]}\to\dbU^{[n,p]}$, and the closed normal subgroup $M=S^{(n,p)}$ of $S$.
We write $\theta_{w,n},\alp_{w,n}$ for $\theta,\alp$, respectively.

\begin{lem}
For $\sig\in S^{(n,p)}$ and a word $w\in X^*$ of length $1\leq s\leq n$ one has
\[
(\sig,\alp_{w,n})_n=\iota_{n,s}(\eps_{w,\dbZ/p^{n-s+1}}(\sig)).
\]
\end{lem}
\begin{proof}
By (\ref{alp and del}) and (\ref{theta and rho}),
\[
(\sig,\alp_{w,n})_n=\theta_{w,n}(\sig)=\iota_{n,s}^\dbU(\rho_{\dbZ/p^{n-s+1}}^w(\sig))=\iota_{n,s}(\eps_{w,\dbZ/p^{n-s+1}}(\sig)).
\qedhere
\]
\end{proof}

This and (\ref{description of pairing}) give:

\begin{cor}
\label{equality of the two pairings}
Let $w,w'$ be words in $X^*$ of lengths $1\leq s,s'\leq n$, respectively, with $w$ Lyndon.
Then
\[
(\tau_w^{p^{n-s}},\alp_{w',n})_n=\langle w,w'\rangle_n.
\]
\end{cor}

Proposition \ref{vanishing of iota}(a)(b) now gives:

\begin{cor}
\label{upper-triangular}
Let  $\Lyn_{\leq n}(X)$ be totally ordered by $\preceq$.
The matrix
\[
\Bigl((\tau_w^{p^{n-|w|}},\alp_{w',n})_n\Bigr),
\]
where $w,w'\in \Lyn_{\leq n}(X)$, is upper-triangular unipotent.
\end{cor}

In general the above matrix need not be the identity matrix -- see e.g., Proposition \ref{uuu} below.
Next we observe the following general fact:

\begin{lem}
\label{linear algebra lemma}
Let $R$ be a commutative ring and let $(\cdot,\cdot)\colon A\times B\to R$ be a non-degenerate bilinear map of $R$-modules.
Let $(L,\leq)$ be a finite totally ordered set, and for every $w\in L$ let $a_w\in A$, $b_w\in B$.
Suppose that the matrix $\bigl((a_w,b_{w'})\bigr)_{w,w'\in L}$ is invertible, and that $a_w$, $w\in L$, generate $A$.
Then $a_w$, $w\in L$, is an $R$-linear basis of $A$, and $b_w$, $w\in L$, is an $R$-linear basis of $B$.
\end{lem}
\begin{proof}
Let $b\in B$, and consider $r_{w'}\in R$, with $w'\in L$.
The assumptions imply that $b=\sum_{w'}r_{w'}b_{w'}$ if and only if  $(a_w,b-\sum_{w'}r_{w'}b_{w'})=0$ for every $w$.
Equivalently, the $r_{w'}$ solve the linear system $\sum_{w'}(a_w,b_{w'})X_{w'}=(a_w,b)$, for $w\in L$.
By the invertibility, the latter system has a unique solution.
This shows that $b_w$, $w\in L$, is an $R$-linear basis of $B$.

By reversing the roles of $a_w,b_w$, we conclude that the $a_w$, $w\in L$, form an $R$-linear basis of $A$.
\end{proof}

\begin{thm}
\label{basis}
\begin{enumerate}
\item[(a)]
The cohomology elements $\alp_{w,n}$, where $w\in\Lyn_{\leq n}(X)$, form a $\dbZ/p$-linear basis of $H^2(S^{[n,p]},\dbZ/p)$.
\item[(b)]
When $X$ is finite, the cosets of the powers $\tau_w^{p^{n-s}}$, $w\in\Lyn_{\leq n}(X)$, form a basis of the $\dbZ/p$-module $S^{(n,p)}/S^{(n+1,p)}$.
\end{enumerate}
\end{thm}
\begin{proof}
When $X$ is finite, the set $\Lyn_{\leq n}(X)$ is also finite.
By Theorem \ref{generators}, the cosets $a_w$ of $\tau_w^{p^{n-|w|}}$, where $w\in \Lyn_{\leq n}(X)$, generate the $\dbZ/p$-module $S^{(n,p)}/S^{(n+1,p)}$.
We apply Lemma \ref{linear algebra lemma} with the $\dbZ/p$-modules $A=S^{(n,p)}/S^{(n+1,p)}$ and $B=H^2(S^{[n,p]},\dbZ/p)$, the non-degenerate bilinear map $A\times B\to\dbZ/p$ induced by $(\cdot,\cdot)_n$, the generators $a_w$ of $A$, and the elements $b_w=\alp_{w,n}$ of $B$.

Corollary \ref{upper-triangular} implies that the matrix $(a_w,b_{w'})$ is invertible.
Therefore Lemma \ref{linear algebra lemma} gives both assertions in the finite case.

The general case of (a) follows from the finite case by a standard limit argument.
\end{proof}

We call $\alp_{w,n}$, $w\in\Lyn_{\leq n}(X)$, the {\bf Lyndon basis} of $H^2(S^{[n,p]},\dbZ/p)$.

Recall that the number of relations in a minimal presentation of a pro-$p$ group $G$ is given by $\dim H^2(G,\dbZ/p)$ \cite{NeukirchSchmidtWingberg}*{Cor.\ 3.9.5}.
In view of (\ref{size of Lyn}), Theorem \ref{basis} gives this number for $G=S^{[n,p]}$:

\begin{cor}
\label{dim of H2}
One has
\[
\dim_{\dbF_p}H^2(S^{[n,p]},\dbZ/p)=\dim_{\dbF_p}(S^{(n,p)}/S^{(n+1,p)})=\sum_{s=1}^n\varphi_s(|X|),
 \]
where $\varphi_s$ is the necklace map.
\end{cor}

\section{The shuffle relations}
We recall the following constructions from \cite{ChenFoxLyndon58},  \cite{Reutenauer93}*{pp.\ 134--135}.
Let $u_1\nek u_t\in X^*$ be words of lengths $s_1\nek s_t$, respectively.
We say that a word $w\in  X^*$ of length $1\leq n\leq s_1+\cdots+s_t$ is an {\bf infiltration} of $u_1\nek u_t$, if there exist sets $I_1\nek I_t$ of respective cardinalities $s_1\nek s_t$ such that $\{1,2\nek n\}=I_1\cup\cdots \cup I_t$ and the restriction of $w$ to the index set $I_j$ is $u_j$, $j=1,2\nek t$.
We then write $w=w(I_1\nek I_t,u_1\nek u_t)$.
We write $\Infil(u_1\nek u_t)$ for the set of all infiltrations of $u_1\nek u_t$.
The {\bf  infiltration product}  $u_1\downarrow\cdots\downarrow u_t$ of $u_1\nek u_t$ is the polynomial $\sum w$ in $\dbZ\langle X\rangle$, where the sum is over all such infiltrations, taken with multiplicity.

If in the above setting, the sets $I_1\nek I_t$ are pairwise disjoint, then $w(I_1\nek I_t,u_1\nek u_t)$ is called a {\bf shuffle} of $u_1\nek u_t$.
We write  $\Sh(u_1\nek u_t)$ for the set of all shuffles of $u_1\nek u_t$.
It consists of the words in $\Infil(u_1\nek u_t)$ of length $s_1+\cdots+s_t$.
The {\bf  shuffle product}  $u_1\sha\cdots\sha u_t$ is the polynomial $\sum w(I_1\nek I_t,u_1\nek u_t)$ in $\dbZ\langle X\rangle$, where the sum is over all shuffles of $u_1\nek u_t$, taken with multiplicity.
Thus $u_1\sha\cdots\sha u_t$ is the homogenous part of  $u_1\downarrow\cdots\downarrow u_t$ of (maximal) degree $s_1+\cdots+s_t$.
For instance
\[
\begin{split}
(xy)\downarrow(xz)&=(xyxz)+2(xxyz)+2(xxzy)+(xzxy)+(xyz)+(xzy), \quad \\
(xy)\sha(xz)&=(xyxz)+2(xxyz)+2(xxzy)+(xzxy)\\
(x)\downarrow(x)&=2(xx)+(x), \quad (x)\sha(x)=2(xx).
\end{split}
\]

We may view infiltration and shuffle products also as elements of $\dbZ_p\langle X\rangle$.
Let $\Shuffles(X)$ be the $\dbZ$-submodule of $\dbZ\langle X\rangle$ generated by all shuffle products $u\sha v$, with $\emptyset \neq u,v\in X^*$.
Let $\Shuffles_n(X)$ be its homogenous component of degree $n$.

\begin{examples}
\label{Shuffles for small n}
$\Shuffles_1(X)=\{0\}$,
\[
\begin{split}
\Shuffles_2(X)&=\langle (xy)+(yx)\ |\ x,y\in X\rangle,\\
\Shuffles_3(X)&=\langle(xyz)+(xzy)+(zxy)\ |\ x,y,z\in X\rangle.
\end{split}
\]
\end{examples}

Let $(\cdot,\cdot)$ be the pairing of (\ref{bilinear}) for the ring $R=\dbZ_p$.
As before, $S=S_X$ is the free profinite group on the set $X$.
The following fact is due to Chen, Fox, and Lyndon in discrete case \cite{ChenFoxLyndon58}*{Th.\ 3.6} (see also \cite{Morishita12}*{Prop.\ 8.6}, \cite{Reutenauer93}*{Lemma 6.7}), as well as \cite{Vogel05}*{Prop.\ 2.25} in the profinite case.

\begin{prop}
\label{CFL}
For every $\emptyset\neq u,v\in X^*$ and every $\sig\in S$ one has
\[
\eps_{u,\dbZ_p}(\sig)\eps_{v,\dbZ_p}(\sig)=(\Lam_{\dbZ_p}(\sig),u\downarrow v).
\]
\end{prop}

\begin{cor}
\label{Lam and shuffles}
Let $u,v$ be nonempty words in $X^*$ with $s=|u|+|v|\leq n$.
For every $\sig\in S^{(n,p)}$ one has $(\Lam_{\dbZ_p}(\sig),u\sha v)\in p^{n-s+1}\dbZ_p$.
\end{cor}
\begin{proof}
If $w$ is a nonempty word of length $|w|<s$, then by Proposition \ref{Koch}(b),
$\eps_{w,\dbZ_p}(\sig)\in p^{n-|w|}\dbZ_p\subseteq p^{n-s+1}\dbZ_p$.
In particular, this is the case for $w=u$, $w=v$, and when $w\in\Infil(u,v)\setminus\Sh(u,v)$.
It follows from Proposition \ref{CFL} that $(\Lam_{\dbZ_p}(\sig),u\sha v)\in p^{n-s+1}\dbZ_p$.
\end{proof}

We obtain the following {\bf shuffle relations} (see also \cite{VogelThesis}*{Cor.\ 1.2.10} and \cite{FennSjerve84}*{Th.\ 6.8}).
We write $X^s$ for the set of words in $X^*$ of length $s$.

\begin{thm}
\label{shuffle relations}
For every $\emptyset\neq  u,v\in X^*$ with $s=|u|+|v|\leq n$ one has
\[
\sum_{w\in X^s}(u\sha v)_w\alp_{w,n}=0.
\]
\end{thm}
\begin{proof}
For  $\sig\in S^{(n,p)}$, Corollary \ref{Lam and shuffles} gives
\[
\sum_{w\in X^s}(u\sha v)_w\eps_{w,\dbZ_p}(\sig)=\sum_{w\in X^*}(u\sha v)_w\eps_{w,\dbZ_p}(\sig)=(\Lam_{\dbZ_p}(\sig),u\sha v)\in p^{n-s+1}\dbZ_p.
\]
Therefore, by Lemma \ref{description of pairing 2},
\[
\begin{split}
(\sig,\sum_{w\in X^s}(u\sha v)_w\alp_{w,n})_n&=\sum_{w\in X^s}(u\sha v)_w(\sig,\alp_{w,n})_n
=\sum_{w\in X^s}(u\sha v)_w\,\iota_{n,s}(\eps_{w,\dbZ/p^{n-s+1}}(\sig))\\
&=\iota_{n,s}\Bigl(\sum_{w\in X^s}(u\sha v)_w\eps_{w,\dbZ/p^{n-s+1}}(\sig)\Bigr)=0.
\end{split}
\]
Now use the fact that  $(\cdot,\cdot)_n\colon S^{(n,p)}\times H^2(S^{[n,p]},\dbZ/p)\to\dbZ/p$ has a trivial right kernel.
\end{proof}

\begin{cor}
\label{structure theorem}
There is a canonical epimorphism
\[
\begin{split}
\bigoplus_{s=1}^n\Bigl(\bigl(\bigoplus_{w\in X^s} \dbZ\bigr)/&\Shuffles_s(X)\Bigr)\tensor(\dbZ/p)\to H^2(S^{[n,p]},\dbZ/p) \\
& (\bar r_w)_w\mapsto\sum_w r_w\alp_{w,n}.
\end{split}
\]
\end{cor}
\begin{proof}
By Theorem \ref{shuffle relations} this homomorphism is well defined.
By Theorem \ref{basis}(a), it is surjective.
\end{proof}

\begin{rem}
\rm
In view of Lemma \ref{description of pairing 2}, the epimorphism of Corollary \ref{structure theorem}  and the canonical pairing $(\cdot,\cdot)_n$ induce a bilinear map
\[
\begin{split}
S^{(n,p)} &\times \bigoplus_{s=1}^n\Bigl(\bigl(\bigoplus_{w\in X^s}\dbZ\bigr)/\Shuffles_s(X)\Bigr)\to \dbZ/p, \quad \\
&(\sig,\overline{(r_w)}_w)=\sum_wr_w\iota_{n,s}(\eps_{w,\dbZ/p^{n-s+1}}(\sig))
\end{split}
\]
with left kernel $S^{(n+1,p)}$.
\end{rem}

\begin{exam}
\label{example}
\rm
We show that for every $x_1,x_2\nek x_k\in X$ one has
\[
(x_1x_2\cdots x_k)+(-1)^{k-1}(x_k\cdots x_2x_1)\in\Shuffles_n(X).
\]
We may assume that $x_1,x_2\nek x_k$ are distinct.
For $1\leq l\leq k-1$ let $u_l=(x_l\cdots x_2x_1)$ and $v_l=(x_{l+1}\cdots x_k)$.
We consider the polynomial $\sum_{l=1}^{k-1}(-1)^{l-1}u_l\sha v_l$ in $\dbZ\langle X\rangle$.
It is homogenous of degree $k$.
If $w\in\Sh(u_l,v_l)$, then either:
\begin{enumerate}
\item
$x_l$ appears before $x_{l+1}$ in $w$, and then $w$ appears with an opposite sign also in $\Sh(u_{l-1},v_{l-1})$; or
\item
$x_{l+1}$ appears before $x_l$ in $w$, and then $w$ appears with opposite sign also in $\Sh(u_{l+1},v_{l+1})$.
\end{enumerate}
The only exceptions are  $w=(x_1x_2\cdots x_k)\in\Sh(u_1,v_1)$ and $w=(x_k\cdots x_2x_1)\in \Sh(u_{k-1},v_{k-1})$.
This shows that
\[
(x_1x_2\cdots x_k)+(-1)^k(x_k\cdots x_2x_1)=\sum_{l=1}^{k-1}(-1)^{l-1}u_l\sha v_l.
\]
For $1\leq k\leq n$ Corollary \ref{structure theorem} therefore implies that
\[
\alp_{(x_1x_2\cdots x_k),n}=(-1)^{k-1}\alp_{(x_k\cdots x_2x_1),n}.
\]
\end{exam}

\section{Example: The case $n=2$}
\label{section on n=2}
Our results in this case are fairly well known, and are brought here in order to illustrate the general theory.

As before let $S=S_X$ with $X$ totally ordered.
Here $S^{(2,p)}=S^p[S,S]$ and $\bar S=S^{[2,p]}$ is the maximal elementary $p$-abelian quotient of $S$.
We may identify $H^1(S,\dbZ/p)=H^1(\bar S,\dbZ/p)\isom\bigoplus_{x\in X}\dbZ/p$.
Let $\chi_{x,\dbZ/p}=\eps_{(x),\dbZ/p}$, $x\in X$, be the basis of $H^1(S,\dbZ/p)$ dual to $X$ (see Remark \ref{dual basis}).

The Lyndon words $w$ of length $\leq2$ are $(x)$, where $x\in X$, and $(xy)$, where $x,y\in X$ and $x<y$.
For these words we have $\tau_{(x)}=x$ and $\tau_{(xy)}=[x,y]$.
By Examples \ref{examples of  alp}(1)(3),
\[
\alp_{(x),2}=\Bock_{p,\bar S}(\chi_{x,\dbZ/p}), \quad \alp_{(xy),2}=\chi_{x,\dbZ/p}\cup\chi_{y,\dbZ/p}.
\]
Hence, by Theorem \ref{basis}, $\Bock_{p,\bar S}(\chi_{x,\dbZ/p})$ and $\chi_{x,\dbZ/p}\cup\chi_{y,\dbZ/p}$, where $x<y$, form a $\dbZ/p$-linear basis of $H^2(\bar S,\dbZ/p)$.
Furthermore, when $X$ is finite, the elements of the form $x^p$ and $[x,y]$ with $x,y\in X$, $x<y$, form a $\dbZ/p$-linear basis of $S^{(2,p)}/S^{(3,p)}$.
In view of Examples \ref{Shuffles for small n} and Corollary \ref{structure theorem}, the map $(\bar r_w)\mapsto\sum_w r_w\alp_{w,2}$ induces an epimorphism
\[
\bigoplus_{x\in X}\dbZ/p\oplus\Bigl(\bigl(\bigoplus_{x,y\in X}\dbZ\bigr)/\langle (xy)+(yx)\ |\ x,y\in X\rangle\Bigr)\tensor(\dbZ/p)\to H^2(\bar S,\dbZ/p).
\]
It coincides with the map
\[
H^1(\bar S,\dbZ/p)\oplus{\textstyle\bigwedge^2H^1(\bar S,\dbZ/p)}\to H^2(\bar S,\dbZ/p)
\]
which is $\Bock_{p,\bar S}$ on the first component and $\cup$ on the second component.
When $p\neq2$ the direct sum is a free $\dbZ/p$-module on $\Lyn_{\leq2}(X)$, and by comparing dimensions we see that the epimorphism is in fact an isomorphism (compare \cite{EfratMinac11}*{Cor.\ 2.9(a)}).
However when $p=2$ one has $\Bock_{2,\bar S}(\chi)=\chi\cup\chi$ \cite{EfratMinac11}*{Lemma 2.4}, so the above epimorphism is not injective.

Next, Proposition \ref{vanishing of iota} shows that the matrix $(\langle w,w'\rangle_2)$, where $w,w'\in\Lyn_{\leq2}(X)$, is the identity matrix.
In view of Corollary \ref{equality of the two pairings}, it coincides with the matrix $\bigl((\tau_w^{p^{2-|w|}},\alp_{w',2})_2\bigr)$.
Thus
\[
\begin{split}
&(x^p,\Bock_{p,\bar S}(\chi_{x,\dbZ/p}))_2=1\hbox{ for every } x\in X ,\\
&(x^p,\Bock_{p,\bar S}(\chi_{y,\dbZ/p}))_2=0 \hbox{ for every } x,y\in X, x\neq y,\\
&(x^p,\chi_{y,\dbZ/p}\cup\chi_{z,\dbZ/p})_2=0 \hbox{ for every } x,y,z\in X,\\
&([x,y],\Bock_{p,\bar S}(\chi_{z,\dbZ/p}))_2=0 \hbox{ for every } x,y,z\in X,\\
&([x,y],\chi_{z,\dbZ/p}\cup\chi_{t,\dbZ/p})_2=0 \hbox{ for every } x,y,z,t\in X, (xy)\neq(zt),(tz),\\
&([x,y],\chi_{x,\dbZ/p}\cup\chi_{y,\dbZ/p})_2=1 \hbox{ for every } x,y\in X \hbox{ with } x<y.
\end{split}
\]
This recovers well known facts from \cite{Labute66}*{\S2.3},
\cite{Koch02}*{\S7.8} and \cite{NeukirchSchmidtWingberg}*{Th.\ 3.9.13 and  Prop.\ 3.9.14}

\section{Example: The case $n=3$.}
\label{section on n=3}
Here $S^{(3,p)}=S^{p^2}[S,S]^p[S,[S,S]]$.
We abbreviate $\bar S=S^{[3,p]}$.
Recall that $\Lyn_{\leq3}(X)$ consists of the words
\[
\begin{split}
&(x) \hbox{ for } x\in X,\\
& (xy), (xxy), (xyy) \hbox{ for } x,y\in X \hbox{ with } x<y, \\
&  (xyz), (xzy) \hbox{  for } x,y,z\in X \hbox{ with }  x<y<z.
\end{split}
\]
For these words
\[
\begin{split}
\tau_{(x)}=x, \quad
&\tau_{(xy)}=[x,y],\quad
\tau_{(xxy)}=[x,[x,y]], \quad
\tau_{(xyy)}=[[x,y],y],\quad \\
&\tau_{(xyz)}=[x,[y,z]], \quad
\tau_{(xzy)}=[[x,z],y].
\end{split}
\]

By Theorem \ref{generators}, the cosets of
\[
x^{p^3},\  [x,y]^p,\  [x,[x,y]],\ [[x,y],y],\ [x,[y,z]],\ [[x,z],y],
\]
with $x,y,z$ as above, generate $S^{(3,p)}/S^{(4,p)}$.
When $X$ is finite, they form a linear basis of  $S^{(3,p)}/S^{(4,p)}$ over $\dbZ/p$ (Theorem \ref{basis}(b)).
Furthermore,  Theorem \ref{basis}(a) gives:

\begin{thm}
The following cohomology elements form a $\dbZ/p$-linear basis of $H^2(\bar S,\dbZ/p)$:
\[
\alp_{(x),3}, \ \alp_{(xy),3},\ \alp_{(xxy),3},\ \alp_{(xyy),3},\ \alp_{(xyz),3},\ \alp_{(xzy),3},
\]
where $x,y,z\in X$  and we assume that $x<y<z$.
\end{thm}
By Examples \ref{examples of  alp}, $\alp_{(x),3}=\Bock_{p^2,\bar S}(\chi_{x,\dbZ/p^2})$, and for every $x,y,z\in X$, $\alp_{(xyz),3}$ belongs to the triple Massey product $\langle\chi_{x,\dbZ/p},\chi_{y,\dbZ/p},\chi_{z,\dbZ/p}\rangle\subseteq H^2(\bar S,\dbZ/p)$.

We further recall that $\alp_{(xy),3}$ is the pullback to $H^2(\bar S,\dbZ/p)$ under $\bar\rho^{(xy)}_{\dbZ/p^2}\colon \bar S\to\dbU_3(\dbZ/p^2)^{[3,p]}$ of the cohomology element in $H^2(\dbU_3(\dbZ/p^2)^{[3,p]},\dbZ/p)$ corresponding to the central extension
\[
0\to\dbZ/p\to \dbU_3(\dbZ/p^2)\to \dbU_3(\dbZ/p^2)^{[3,p]}\to1.
\]
Alternatively, it has the following explicit description:
By Proposition \ref{properties of dbU}(a), $ \dbU_3(\dbZ/p^2)^{(3,p)}=I_3+\dbZ pE_{13}$, and let $\iota=\iota_{3,2}^\dbU\colon \dbU_3(\dbZ/p^2)^{(3,p)}\xrightarrow{\sim}\dbZ/p$ be the natural isomorphism.
By (\ref{theta and rho}) and (\ref{alp and del}), $\alp_{(xy),3}=-\trg(\theta)$, where $\theta=\iota\circ(\rho^{(xy)}_{\dbZ/p^2}|_{S^{(3,p)}})$ and
$\trg\colon H^1(S^{(3,p)},\dbZ/p)^S\xrightarrow{\sim} H^2(\bar S,\dbZ/p)$ is the transgression isomorphism.

Next we compute the matrix $((\tau_w^{p^{3-|w|}},\alp_{w',3})_3)=(\langle w,w'\rangle_3)$, where $w,w'\in\Lyn_{\leq3}(X)$:

\begin{prop}
\label{uuu}
For $w,w'\in\Lyn_{\leq3}(X)$ one has
\[
\langle w,w'\rangle_3=
\begin{cases}
\ \ 1, &\hbox{if } w=w';\\
-1,&\hbox{if } w=(xyz), \ w'=(xzy)  \hbox{ for some } x,y,z\in X, \ x<y<z;
\\
\ \ 0,&\hbox{otherwise}.
\end{cases}
\]
\end{prop}
\begin{proof}
In view of  Proposition \ref{vanishing of iota}, it is enough to show the assertion when $w\neq w'$, either $|w|=|w'|$ or $2|w|\leq|w'|$, and the letters of $w'$ appear in $w$.
Furthermore, when $|w|=|w'|$ we may assume that $w\leq_{\rm alp} w'$.

Thus when $w$ has one of the forms $(x),(xy),(xyy),(xzy)$ (where $x<y<z$) there is nothing more to show.

When $w=(xxy)$ with $x<y$ we need to check only the word $w'=(xyy)$.
Then  Lemma \ref{commutators} gives
\[
\begin{split}
\langle(xxy),(xyy)\rangle_3&=\eps_{(xyy),\dbZ/p}([x,[x,y]])\\
&=\eps_{(x),\dbZ/p}(x)\cdot\eps_{(yy),\dbZ/p}([x,y])-\eps_{(y),\dbZ/p}(x)\cdot\eps_{(xy),\dbZ/p}([x,y])\\
&=1\cdot0-0\cdot1=0.
\end{split}
\]

When $w=(xyz)$ with $x<y<z$  we need to check only the word $w'=(xzy)$.
Then  Lemma \ref{commutators} gives
\[
\begin{split}
\langle(xyz),(xzy)\rangle_3&=\eps_{(xzy),\dbZ/p}([x,[y,z]])\\
&=\eps_{(x),\dbZ/p}(x)\cdot\eps_{(zy),\dbZ/p}([y,z])-\eps_{(y),\dbZ/p}(x)\cdot\eps_{(xz),\dbZ/p}([y,z])\\
&=1\cdot(-1)-0\cdot0=-1.
\end{split}
\]
This completes the verification in all cases.
\end{proof}

In view of Examples \ref{Shuffles for small n}, Corollary \ref{structure theorem} gives rise to an epimorphism
\begin{equation}
\label{big epimorphism}
\begin{split}
\bigoplus_{x\in X}\dbZ/p&\oplus\Bigl(\bigl(\bigoplus_{x,y\in X}\dbZ(xy)\bigr)/\langle (xy)+(yx)\ |\ x,y\in X\rangle\Bigr)\tensor(\dbZ/p)\\
&\oplus\Bigl(\bigl(\bigoplus_{x,y,z\in X}\dbZ(xyz)\bigr)/\langle (xyz)+(xzy)+(zxy)\ |\ x,y,z\in X\rangle\Bigr)\tensor(\dbZ/p)\\
&\to H^2(\bar S,\dbZ/p).
\end{split}
\end{equation}
Moreover, for $x,y,z\in X$, $x<y<z$, we have
\[
\begin{split}
&(yx)=(x)\sha(y)-(xy)\\ 
&2(xx)=(x)\sha(x)\\ 
&(xyx)=(x)\sha(xy)-2(xxy)\\ 
&(yxx)=(x)\sha(yx)-(xx)\sha(y)+(xxy)\\
&(yxy)=(xy)\sha(y)-2(xyy)\\
&(yyx)=(yy)\sha(x)-(y)\sha(xy)+(xyy)\\
&(yxz)=(y)\sha(xz)-(xyz)-(xzy)\\
&(zxy)=(z)\sha(xy)-(xzy)-(xyz)\\
&(yzx)=(zx)\sha(y)-(x)\sha(zy)+(xzy)\\
&(zyx)=(yx)\sha(z)-(x)\sha(yz)+(xyz)\\
&3(xxx)=(x)\sha(xx). 
\end{split}
\]
These congruences and Example \ref{examples of Lyndon words} imply that
\[
\begin{split}
&\sum_{w\in X^2}\dbZ w\equiv \sum_{w\in \Lyn_2(X)}\dbZ w+\hbox{ 2-torsion}
\pmod{\Shuffles_2(X)},\\
&\sum_{w\in X^3}\dbZ w\equiv \sum_{w\in \Lyn_3(X)}\dbZ w+\hbox{ 3-torsion} \pmod{\Shuffles_3(X)}.\\
\end{split}
\]
Therefore, for $p>3$, the direct sum in (\ref{big epimorphism}) is the free $\dbZ/p$-module on the basis $\Lyn_{\leq3}(X)$.
Thus the epimorphism (\ref{big epimorphism})  maps the $\dbZ/p$-linear basis $1w$, $w\in\Lyn_{\leq3}(X)$, bijectively onto the $\dbZ/p$-linear basis $\alp_{w,3}$, $w\in\Lyn_{\leq3}(X)$ (see Theorem \ref{basis}).
Consequently we have:

\begin{thm}
\label{shuffle relations for n=3}
For $n=3$ and $p>3$,  (\ref{big epimorphism}) is an isomorphism.
Thus all relations in $H^2(S^{[3,p]},\dbZ/p)$ are consequences of the shuffle relations of Theorem \ref{shuffle relations}.
\end{thm}

\end{document}